\documentclass[a4paper,12pt]{article}
\usepackage{amsmath}
\usepackage{amsthm}
\usepackage{amssymb}
\usepackage{enumerate}
\usepackage{url}
\usepackage[utf8]{inputenc}
\usepackage{comment}

\title{A Girsanov-type formula for a class of anticipative transforms of Brownian motion associated with exponential functionals}

\author{Yuu Hariya\thanks{Supported in part by JSPS KAKENHI Grant Number 22K03330}}

\date{\empty}
\numberwithin{equation}{section}

\theoremstyle{plain}

\newtheorem{thm}{Theorem}[section]

\newtheorem{prop}{Proposition}[section]

\newtheorem{cor}{Corollary}[section]

\newtheorem{lem}{Lemma}[section]

\theoremstyle{definition}

\theoremstyle{remark}
\newtheorem{rem}{Remark}[section]
\newtheorem{exm}{Example}[section]

\begin{document}

\newcommand\ND{\newcommand}
\newcommand\RD{\renewcommand}

\ND\N{\mathbb{N}}
\ND\R{\mathbb{R}}
\ND\Q{\mathbb{Q}}
\ND\C{\mathbb{C}}

\ND\F{\mathcal{F}}

\ND\kp{\kappa}

\ND\ind{\boldsymbol{1}}

\ND\al{\alpha }
\ND\la{\lambda }
\ND\La{\varLambda }
\ND\ve{\varepsilon}
\ND\Om{\Omega}

\ND\ga{\gamma}

\ND\lref[1]{Lemma~\ref{#1}}
\ND\tref[1]{Theorem~\ref{#1}}
\ND\pref[1]{Proposition~\ref{#1}}
\ND\sref[1]{Section~\ref{#1}}
\ND\ssref[1]{Subsection~\ref{#1}}
\ND\aref[1]{Appendix~\ref{#1}}
\ND\rref[1]{Remark~\ref{#1}} 
\ND\cref[1]{Corollary~\ref{#1}}
\ND\eref[1]{Example~\ref{#1}}
\ND\fref[1]{Fig.\ {#1} }
\ND\lsref[1]{Lemmas~\ref{#1}}
\ND\tsref[1]{Theorems~\ref{#1}}
\ND\dref[1]{Definition~\ref{#1}}
\ND\psref[1]{Propositions~\ref{#1}}
\ND\rsref[1]{Remarks~\ref{#1}}
\ND\sssref[1]{Subsections~\ref{#1}}

\ND\pr{\mathbb{P}}
\ND\ex{\mathbb{E}}

\ND\eb[1]{e^{B_{#1}}}
\ND\ebm[1]{e^{-B_{#1}}}

\ND\vp{\varphi}
\ND\eqd{\stackrel{(d)}{=}}
\ND\db[1]{B^{(#1)}}
\ND\dcb[1]{\cB ^{(#1)}}
\ND\da[1]{A^{(#1)}}
\ND\dca[1]{\cA ^{(#1)}}
\ND\dz[1]{Z^{(#1)}}
\ND\Z{\mathcal{Z}}

\ND\Ga{\Gamma}

\ND\tr{\mathbb{T}}

\ND\ct{\mathcal{T}}

\ND\id{\mathrm{Id}}

\ND\cS{\mathcal{S}}

\ND\h{\mathfrak{h}}

\ND{\rmid}[1]{\mathrel{}\middle#1\mathrel{}}

\def\thefootnote{{}}

\maketitle 
\begin{abstract}
In this paper, with the help of a result by Matsumoto--Yor (2000), we prove a Girsanov-type formula for a class of anticipative transforms of Brownian motion which possesses exponential functionals as anticipating factors. Our result unifies existing formulas in earlier works. As an application, we also consider the law of Brownian motion perturbed by a positive weight of a fairly wide class, and prove its invariance under an anticipative transformation associated with the perturbation. In the course of our exploration, a disintegration formula for the Wiener measure related to exponential functionals plays a key role.
\footnote{Mathematical Institute, Tohoku University, Aoba-ku, Sendai 980-8578, Japan}
\footnote{E-mail: hariya@tohoku.ac.jp}
\end{abstract}

\section{Introduction and main results}\label{;intro}
A Girsanov-type formula, or a change of measure formula, 
for anticipative transforms of Brownian motion has been studied 
by a number of authors, especially in the framework of 
Malliavin calculus; see, e.g., \cite{buc,kus,ram,uz94,yano,zz} and references therein. 
In the formula, the density with respect to the underlying 
Wiener measure is given by the product of two factors, 
one of which is a stochastic exponential in which the It\^o 
integral is replaced by the Skorokhod integral, and the other is 
a Carleman--Fredholm determinant; in some specific settings, 
further factorizations of these two factors have 
also been investigated. As far as we know, there seem to be 
not so many concrete examples in which the corresponding 
densities, in particular, Carleman--Fredholm determinants, 
are explicitly calculated.

In this paper, with the help of a result by Matsumoto--Yor \cite{myPI} 
on exponential functionals of Brownian motion, we introduce 
a class of anticipative transformations under which we can obtain 
a Girsanov-type formula in an explicit form; 
we expect that the result provides a number of examples in 
Malliavin calculus in which we are able to calculate 
Carleman--Fredholm determinants explicitly. We also apply 
it to derive the distributional invariance of Brownian motion 
perturbed by a positive weight of a wide class, which is 
described in terms of an anticipative transformation determined 
from the perturbation. As will be seen below, properties of 
a certain class of anticipative transformations investigated in 
\cite{har22}, as well as a disintegration formula for the Wiener 
measure, play an essential role in the course of our exploration; 
see \lref{;lttrans} and \pref{;pdisint}.

To state the main results of the paper, we prepare some of the 
notation. Let $B=\{ B_{t}\} _{t\ge 0}$ be a one-dimensional standard 
Brownian motion. Let $C([0,\infty );\R )$ be the space of 
continuous functions $\phi :[0,\infty )\to \R $, on which we 
define the transformation
\begin{align}
 A_{t}(\phi ):=\int _{0}^{t}e^{2\phi _{s}}\,ds,\quad t\ge 0; \label{;defa}
\end{align}
with slight abuse of notation, we will simply write $A_{t}$ for 
$A_{t}(B)$: $A_{t}=A_{t}(B)$. This exponential additive 
functional $A_{t},\,t\ge 0$, which is the quadratic variation of 
the geometric Brownian motion $e^{B_{t}},\,t\ge 0$, appears 
in a number of areas in probability theory such as mathematical 
finance and diffusion processes in random environments, and 
is known for its close relationship with planar Brownian motion 
(or two-dimensional Bessel process); see the detailed surveys 
\cite{mySI, mySII} by Matsumoto and Yor. 
Following the notation in \cite{dmy}, we also define 
\begin{align}\label{;defz}
 Z_{t}(\phi ):=e^{-\phi _{t}}A_{t}(\phi ),\quad t\ge 0,
\end{align}
for $\phi \in C([0,\infty );\R )$, and denote $Z_{t}(B)$ by 
$Z_{t}$ for simplicity, too. Given $t>0$, we restrict the 
transformation $A$ to the space $C([0,t];\R )$ of real-valued 
continuous functions over $[0,t]$, and recall from \cite{har22} 
the family $\{ \tr ^{}_{z}\} _{z\in \R }$ of anticipative path 
transformations on $C([0,t];\R )$ defined by 
\begin{align}\label{;ttrans}
 \tr _{z}(\phi )(s)\equiv \tr ^{t}_{z}(\phi )(s):=\phi _{s}-\log \left\{ 
 1+\frac{A_{s}(\phi )}{A_{t}(\phi )}\left( e^{z}-1\right) 
 \right\} ,\quad 0\le s\le t, 
\end{align}
for $\phi \in C([0,t];\R )$. In what follows, with $t>0$ fixed, 
we suppress the superscript $t$ from the notation and 
suppose that each $\tr _{z}$ acts on $C([0,t];\R )$. 
We denote by $C([0,t];\R ^{2})$ the space of $\R ^{2}$-valued 
continuous functions over $[0,t]$. One of the main results 
of the paper is then stated as 
\begin{thm}\label{;tmain}
 Let $\h \equiv \h (t,\,\cdot \,):C([0,t];\R )\to \R $ be a measurable 
 function such that, for every $\phi \in C([0,t];\R )$, the function 
 $\h _{\phi }:\R \to \R $ defined by 
 \begin{align*}
  \h _{\phi }(\xi ):=\h \bigl( 
  \tr _{\phi _{t}-\xi }(\phi )
  \bigr) ,\quad \xi \in \R , 
 \end{align*}
 is of class $C^{1}$ and strictly monotone. Then, for every nonnegative measurable function $F$ on 
 $C([0,t];\R ^{2})$, we have 
 \begin{equation}\label{;qtmain}
  \begin{split}
  &\ex \biggl[ 
  F\bigl( 
  \tr _{B_{t}-\h (B)}(B),B
  \bigr) \exp \left\{ 
  -\frac{\cosh \h (B)}{Z_{t}}+\frac{\cosh B_{t}}{Z_{t}}
  \right\} \left| \h _{B}'(B_{t})\right| \biggr] \\
  &=\ex \!\left[ 
  F\bigl( 
  B,\tr _{B_{t}-\h _{B}^{-1}(B_{t})}(B)
  \bigr) ;\,B_{t}\in \h _{B}(\R )
  \right] ,
 \end{split}
 \end{equation}
 where, for every $\phi \in C([0,t];\R )$, $\h _{\phi }^{-1}$ denotes the 
 inverse function of $\h _{\phi }$.
\end{thm}

In the above statement, as well as in the sequel, 
we equip $C([0,t];\R )$ with topology 
of uniform convergence, and we say that a real-valued 
function on this space is measurable if it is 
Borel-measurable with respect to the topology. 
The same remark also applies to $C([0,t];\R ^{2})$. 

As listed below, existing formulas may be obtained from 
the above theorem with particular choices of the function 
$\h (\phi ),\,\phi \in C([0,t];\R )$: 
given $\al ,x\ge 0$ and $z\in \R $, 
\begin{itemize}
 \item \cite[Theorem~1.1]{dmy} corresponds to 
 $\h (\phi )=\phi _{t}-\log \left\{ 1+\al Z_{t}(\phi )\right\} $;
 
 \item \cite[Theorem~1.5]{dmy} corresponds to 
 $\h (\phi )=-\log \left\{ e^{-\phi _{t}}+\al Z_{t}(\phi )\right\} $;
 
 \item \cite[Theorem~1.2]{har22} corresponds to 
 $\h (\phi )=\phi _{t}-z$;
 
 \item \cite[Theorem~1.1]{har22a} corresponds to 
 $\h (\phi )=-\phi _{t}$;
 
 \item the former relation in \cite[Proposition~5.3]{har22a} 
 corresponds to 
 \begin{align*}
  \h (\phi )=\log \left\{ e^{-\phi _{t}}+2xZ_{t}(\phi )\right\} ;
 \end{align*}
 
 \item the latter relation in \cite[Proposition~5.3]{har22a} 
 corresponds to 
 \begin{align*}
  \h (\phi )=-\log \left\{ e^{\phi _{t}}+2xZ_{t}(\phi )\right\} .
 \end{align*}
\end{itemize}
In fact, the first two formulas are recovered with slight 
generalization. We also remark that the path transformation 
corresponding to the second case is a non-anticipative one, 
but the resulting formula is not the one that follows from 
Girsanov's formula; details will be found in \ssref{;sset}.

As an application of \tref{;tmain}, we also prove 

\begin{cor}\label{;cmain}
 Let $\La \equiv \La (t,\,\cdot \,)$ be a positive 
 continuous function on $C([0,t];\R )$ such that 
 \begin{align}\label{;cal}
  \int _{\R }d\xi \,\La 
  \bigl( \tr _{\phi _{t}-\xi }(\phi )\bigr) 
  \exp \left\{ 
  -\frac{\cosh \xi }{Z_{t}(\phi )}
  \right\} <\infty \quad \text{for all $\phi \in C([0,t];\R )$}.
 \end{align}
 Then, for every nonnegative measurable function $F$ on 
 $C([0,t];\R ^{2})$, we have 
 \begin{align}\label{;qcmain}
  \ex \!\left[ 
  F\bigl( 
  \tr _{B_{t}-h_{\La }(B_{t},B)}(B),B
  \bigr) \La (B)
  \right] 
  =\ex \!\left[ 
  F\bigl( B,\tr _{B_{t}-h_{\La }(B_{t},B)}(B)\bigr) \La (B)
  \right] ,
 \end{align}
 where $h_{\La }\equiv h_{\La }(t,\,\cdot \,,\,\cdot \,)$ is defined through
 \begin{equation}\label{;dhal}
 \begin{split}
  &\int _{-\infty }^{h_{\La }(\xi ,\phi )}dx\,
  \La \bigl( \tr _{\phi _{t}-x}(\phi )\bigr) 
  \exp \left\{ 
  -\frac{\cosh x}{Z_{t}(\phi )}
  \right\} \\
  &=\int _{\xi }^{\infty }dx\,
  \La \bigl( \tr _{\phi _{t}-x}(\phi )\bigr) 
  \exp \left\{ 
  -\frac{\cosh x}{Z_{t}(\phi )}
  \right\} 
 \end{split} 
 \end{equation}
 for $\xi \in \R $ and $\phi \in C([0,t];\R )$.
\end{cor}

Observe that, for every fixed $\phi \in C([0,t];\R )$, by \eqref{;cal} and 
the positivity of $\La $, the function $h_{\La }(\,\cdot \,,\phi )$ 
is strictly decreasing and satisfies 
\begin{align*}
 \lim _{\xi \to -\infty }h_{\La }(\xi ,\phi )=\infty , && 
 \lim _{\xi \to \infty }h_{\La }(\xi ,\phi )=-\infty .
\end{align*}
If we denote by $\ct _{\La }$ the path transformation 
as in \eqref{;qcmain}, namely, if we set 
\begin{align}\label{;dctal}
 \ct _{\La }(\phi )(s):=\tr _{\phi _{t}-h_{\La }(\phi _{t},\phi )}(\phi )(s),
 \quad 0\le s\le t,
\end{align}
for $\phi \in C([0,t];\R )$, then the above corollary also 
indicates that $\ct _{\La }$ is an involution: 
$\ct _{\La }\circ \ct _{\La }=\id $, which is indeed the case 
as will be noted in \rref{;rinvol2}. Here $\id $ is the 
identity map on $C([0,t];\R )$. 

For every $\mu \in \R $, we denote by 
$\db{\mu }=\bigl\{ \db{\mu }_{s}:=B_{s}+\mu s\bigr\} _{s\ge 0}$ 
the Brownian motion with drift $\mu $, to which we will also 
associate the two processes $\bigl\{ \da{\mu }_{s}\bigr\} _{s\ge 0}$ 
and $\bigl\{ \dz{\mu }_{s}\bigr\} _{s\ge 0}$ in such a way that  
\begin{align*}
 \da{\mu }_{s}:=A_{s}(\db{\mu }), && 
 \dz{\mu }_{s}:=Z_{s}(\db{\mu });
\end{align*}
when $\mu =0$, we suppress it from the notation as 
already introduced above. If we apply 
\cref{;cmain} to the function $\La $ of the form 
\begin{align*}
 \La (\phi )=\exp \left( 
 \mu \phi _{t}-\frac{\mu ^{2}}{2}t
 \right) , \quad \phi \in C([0,t];\R ),
\end{align*}
then, by the Cameron--Martin formula, we see that the law of 
$\bigl\{ \db{\mu }_{s}\bigr\} _{0\le s\le t}$ is 
invariant under $\ct _{\La }$ with the above choice of 
$\La $, which extends \cite[Theorem~1.1]{har22a} 
to the case of Brownian motion with drift; see \ssref{;ssec} 
for more details as well as other examples.

The rest of the paper is organized as follows. 
In \sref{;spr}, we state and prove the two key assertions, 
namely \lref{;lttrans} and \pref{;pdisint} as mentioned earlier, 
with which we prove \tref{;tmain} and \cref{;cmain} in \sref{;sprf}. 
We devote \sref{;se} to examples that are obtained 
by applying \tref{;tmain} and \cref{;cmain}.

\section{Preliminaries}\label{;spr}
In this section, we state and prove \lref{;lttrans} and \pref{;pdisint}. 
We keep $t>0$ fixed and begin with properties of 
$\tr _{z},\,z\in \R $, investigated in \cite{har22}.  
Observe that the two transformations $A$ and $Z$ 
defined respectively by \eqref{;defa} and \eqref{;defz} 
are related via 
\begin{align}\label{;deria}
 \frac{d}{ds}\frac{1}{A_{s}(\phi )}
 =-\left\{ \frac{1}{Z_{s}(\phi )}\right\} ^{2},\quad s>0,
 \ \phi \in C([0,\infty );\R ),
\end{align}
and hence, restricted to $C([0,t];\R )$,
\begin{align}\label{;deriad}
 \frac{1}{A_{s}(\phi )}
 =\int _{s}^{t}\frac{du}{
 \left\{ Z_{u}(\phi )\right\} ^{2}}
 +\frac{e^{-\phi _{t}}}{Z_{t}(\phi )},\quad 0<s\le t, 
\end{align}
for every $\phi \in C([0,t];\R )$, because of the 
relation $A_{t}(\phi )=e^{\phi _{t}}Z_{t}(\phi )$.

\begin{lem}[{\cite[Proposition~2.1]{har22}}]\label{;lttrans}
The transformations $\tr _{z},\,z\in \R $, have the following properties. 
\begin{itemize}
\item[\thetag{i}] For every $z\in \R $ and $\phi \in C([0,t];\R )$,
$\tr _{z}(\phi )(t)=\phi _{t}-z$.

\item[\thetag{ii}] For every $z\in \R $ and $\phi \in C([0,t];\R )$, 
\begin{align}\label{;arecipr}
 \frac{1}{A_{s}(\tr _{z}(\phi ))}=\frac{1}{A_{s}(\phi )}
 +\frac{e^{z}-1}{A_{t}(\phi )},\quad 0<s\le t;
\end{align}
in particular, $A_{t}(\tr _{z}(\phi ))=e^{-z}A_{t}(\phi )$.

\item[\thetag{iii}] $Z\circ \tr _{z}=Z$ for any $z\in \R $.

\item[\thetag{iv}] (Semigroup property) 
$\tr _{z}\circ \tr _{z'}=\tr _{z+z'}$ for any $z,z'\in \R $; in particular, 
\begin{align*}
 \tr _{z}\circ \tr _{-z}=\tr _{0}=\id \quad \text{for any $z\in \R $},
\end{align*}
where $\id $ is the identity map on $C([0,t];\R )$ as referred to 
in \sref{;intro}.
\end{itemize}
\end{lem}

We give below a proof of the above lemma for the reader's convenience.

\begin{proof}[Proof of \lref{;lttrans}]
\thetag{i} By definition, 
\begin{align*}
 \tr _{z}(\phi )(t)&=\phi _{t}-\log \left\{ 
 1+(e^{z}-1)
 \right\} \\
 &=\phi _{t}-z.
\end{align*}

\thetag{ii} For the case $z=0$ is obvious, we let $z\neq 0$ 
and compute, for every $0\le s\le t$, 
\begin{align*}
 A_{s}(\tr _{z}(\phi ))
 &=\int _{0}^{s}du\,\frac{e^{2\phi _{u}}}{\left\{ 
 1+\frac{A_{u}(\phi )}{A_{t}(\phi )}(e^{z}-1)
 \right\} ^{2}}\\
 &=\frac{A_{t}(\phi )}{e^{z}-1}\left\{ 
 1-\frac{1}{1+\frac{A_{s}(\phi )}{A_{t}(\phi )}(e^{z}-1)}
 \right\} \\
 &=\frac{A_{s}(\phi )}{1+\frac{A_{s}(\phi )}{A_{t}(\phi )}(e^{z}-1)}, 
\end{align*}
which entails \eqref{;arecipr}.

\thetag{iii} In view of relation~\eqref{;deria}, taking the derivative 
with respect to $s$ on each side of \eqref{;arecipr} yields 
\begin{align*}
 \left\{ Z_{s}(\tr _{z}(\phi ))\right\} ^{-2}
 =\left\{ Z_{s}(\phi )\right\} ^{-2},\quad 0<s\le t,
\end{align*}
from which the claim follows by the positivity of $Z$.

\thetag{iv} By noting that 
\begin{align*}
 \phi _{s}=\frac{1}{2}\log \frac{d}{ds}A_{s}(\phi ),\quad s>0,
 \ \phi \in C([0,\infty );\R ), 
\end{align*}
it suffices to show that, for each 
$\phi \in C([0,t];\R )$, 
\begin{align}\label{;epp1}
 A_{s}\bigl( (\tr _{z}\circ \tr _{z'})(\phi )\bigr)  
 =A_{s}(\tr _{z+z'}(\phi )),\quad 0\le s\le t.
\end{align}
To this end, first observe that relation~\eqref{;arecipr} 
may be rewritten as 
\begin{align}\label{;deriaz}
 \frac{1}{A_{s}(\tr _{z}(\phi ))}
 =\int _{s}^{t}\frac{du}{
 \left\{ Z_{u}(\phi )\right\} ^{2}}
 +\frac{e^{-\phi _{t}+z}}{Z_{t}(\phi )},\quad 0<s\le t,
\end{align}
thanks to \eqref{;deriad}.  
For every $0<s\le t$, we apply property~\thetag{iii} 
to \eqref{;deriad} successively to see that 
\begin{align*}
 \frac{1}{A_{s}\bigl( (\tr _{z}\circ \tr _{z'})(\phi )\bigr) }
 &=\int _{s}^{t}\frac{du}{\left\{ Z_{u}(\phi )\right\} ^{2}}
 +\frac{e^{-(\tr _{z}\circ \tr _{z'})(\phi )(t)}}{Z_{t}(\phi )}, 
\end{align*}
in which repeated use of property~\thetag{i} yields 
\begin{align*}
 -(\tr _{z}\circ \tr _{z'})(\phi )(t)&=-\tr _{z'}(\phi )(t)+z\\
 &=-\phi _{t}+z'+z.
\end{align*}
This proves \eqref{;epp1} in view of \eqref{;deriaz}.
\end{proof}

\begin{rem}\label{;rtrev}
If we denote by $R$ the operation of time reversal on 
$C([0,t];\R )$: 
\begin{align}\label{;trev}
 R(\phi )(s):=\phi _{t-s}-\phi _{t},\quad 0\le s\le t,\ 
 \phi \in C([0,t];\R ),
\end{align}
then it also holds that 
\begin{align}\label{;comptz}
 R\circ \tr _{z}=\tr _{-z}\circ R
\end{align}
for any $z\in \R $; see 
\cite[Proposition~2.1\thetag{v}]{har22}.
\end{rem}

The following may be regarded as a disintegration formula for the 
Wiener measure on $C([0,t];\R )$ in terms of the transformations 
$\tr _{z},\,z\in \R $, which seems to be new to our knowledge, and 
is of interest in its own right.

\begin{prop}\label{;pdisint}
For every nonnegative measurable function $F$ on $C([0,t];\R )$, one has 
\begin{align}\label{;disint}
 \ex [F(B)]
 =\ex \!\left[ 
 \int _{\R }\frac{d\xi }{2K_{0}\bigl( 1/Z_{t}(\phi )\bigr) }
 F\bigl( 
 \tr _{\phi _{t}-\xi }(\phi )
 \bigr) \exp \left\{ 
 -\frac{\cosh \xi }{Z_{t}(\phi )}
 \right\} \bigg| _{\phi =B}
 \right] ,
\end{align}
where $K_{0}$ is the modified Bessel function of the third kind 
(or the Macdonald function) of order $0$.
\end{prop}

For the modified Bessel functions, refer to, e.g., 
\cite[Section~5.7]{leb}. 
The above proposition is a consequence of the 
next two lemmas, the first one of which is a restatement 
of \cite[Proposition~1.7]{myPI} in the case of Brownian motion 
without drift. We denote by $\{ \Z _{s}\} _{s\ge 0}$ the natural filtration 
of the process $Z=\{ Z_{s}\} _{s\ge 0}$, where, as mentioned 
in \sref{;intro}, $Z_{s}$ refers to $Z_{s}(B)$ for notational 
simplicity. 

\begin{lem}\label{;lcond}
For every nonnegative measurable function $f$ on $\R $, 
we have, a.s., 
\begin{align*}
 \ex \!\left[ 
 f(B_{t})\rmid| \Z _{t}
 \right] 
 =\int _{\R }\frac{d\xi }{2K_{0}(1/\zeta )}\,f(\xi )\exp \left( 
 -\frac{\cosh \xi }{\zeta }
 \right) \bigg| _{\zeta =Z_{t}}.
\end{align*}
\end{lem}

Notice that, by the nonnegativity of $f$, 
the conditional expectation on the left-hand side in the above 
equation is well-defined in view of the conditional monotone 
convergence theorem, regardless of whether $f(B_{t})$ 
is integrable or not; this convention will also be applied later.

\begin{lem}\label{;lexpress}
For every $\phi \in C([0,t];\R )$ and $\xi \in \R $, it holds that 
\begin{align}
 \phi _{s}&=-\log\left\{ 
 Z_{s}(\phi )\int _{s}^{t}\frac{du}{\left( 
 Z_{u}(\phi )
 \right) ^{2}}+\frac{Z_{s}(\phi )}{Z_{t}(\phi )}e^{-\phi _{t}}
 \right\} ,\label{;iipd}\\
 \tr _{\phi _{t}-\xi }(\phi )(s)&=-\log\left\{ 
 Z_{s}(\phi )\int _{s}^{t}\frac{du}{\left( 
 Z_{u}(\phi )
 \right) ^{2}}+\frac{Z_{s}(\phi )}{Z_{t}(\phi )}e^{-\xi }
 \right\} ,\label{;iipx}
\end{align}
for all $0<s\le t$.
\end{lem}

\begin{rem}\label{;rinit}
By tending $s\downarrow 0$, both of the right-hand sides of 
\eqref{;iipd} and \eqref{;iipx} converge to $\phi _{0}$ by the continuity of 
$\phi $ and $\tr _{\phi _{t}-\xi }(\phi )$; note that 
$\tr _{\phi _{t}-\xi }(\phi )(0)=\phi _{0}$ by definition. 
\end{rem}

The former relation in \lref{;lexpress} indicates that, given $\phi \in C([0,t];\R )$, 
we can reconstruct it from $Z(\phi )$ and the terminal value $\phi _{t}$. 
The latter reveals that, in the disintegration formula~\eqref{;disint}, the integrand 
with respect to $\xi $ in the right-hand side is determined only from $Z(\phi )$ 
and $\xi $. Notice that, 
for every $\phi \in C([0,t];\R )$ and $\xi \in \R $,
\begin{align}\label{;twrel}
 \tr _{\phi _{t}-\xi }(\phi )(t)=\xi ,&& 
 Z_{s}\bigl( \tr _{\phi _{t}-\xi }(\phi )\bigr) =Z_{s}(\phi )\quad \text{for all $0\le s\le t$},
\end{align}
by properties~\thetag{i} and \thetag{iii} in \lref{;lttrans}.

\begin{proof}[Proof of \lref{;lexpress}]
As for \eqref{;iipd}, by relation~\eqref{;deriad} and the definition~\eqref{;defz} of 
the transformation $Z$, 
\begin{align*}
 Z_{s}(\phi )\int _{s}^{t}\frac{du}{\left\{ 
 Z_{u}(\phi )
 \right\} ^{2}}+\frac{Z_{s}(\phi )}{Z_{t}(\phi )}e^{-\phi _{t}}
 &=\frac{Z_{s}(\phi )}{A_{s}(\phi )}\\
 &=e^{-\phi _{s}},
\end{align*}
hence the claim. As for \eqref{;iipx}, we apply \eqref{;iipd} to 
$\tr _{\phi _{t}-\xi }(\phi ) \in C([0,t];\R )$ to obtain the desired expression 
owing to \eqref{;twrel}.
\end{proof}

We are in a position to prove \pref{;pdisint}.

\begin{proof}[Proof of \pref{;pdisint}]
In view of \lref{;lexpress}, it suffices to prove the assertion in the case that 
$F$ is of the form $F(\phi )=f(\phi _{t})G(Z(\phi )),\,\phi \in C([0,t];\R )$, 
with $f$ a nonnegative measurable function on $\R $ and $G$ a nonnegative 
measurable function on $C([0,t];\R )$; then, by approximation, the formula extends 
to any nonnegative measurable function of 
$(\phi _{t},Z(\phi )),\,\phi \in C([0,t];\R )$.
To this end, by \lref{;lcond}, we have, conditionally on $\Z _{t}$, 
\begin{align*}
 \ex [F(B)]&=\ex \!\left[ 
 G(Z)\ex [f(B_{t})\rmid| \Z _{t}]
 \right] \\
 &=\ex \!\left[ 
 G(Z)\int _{\R }\frac{d\xi }{2K_{0}(1/\zeta )}\,f(\xi )\exp \left( 
 -\frac{\cosh \xi }{\zeta }
 \right) \bigg| _{\zeta =Z_{t}}
 \right] .
\end{align*}
Thanks to \eqref{;twrel}, the last expression agrees with the right-hand side of 
\eqref{;disint} because, with the above choice of $F$, 
\begin{align*}
 F\bigl( 
 \tr _{\phi _{t}-\xi }(\phi )
 \bigr) =f(\xi )G(Z(\phi ))
\end{align*}
for every $\phi \in C([0,t];\R )$ and $\xi \in \R $.
\end{proof}

\section{Proofs of \tref{;tmain} and \cref{;cmain}}\label{;sprf}
In this section, we prove \tref{;tmain} and \cref{;cmain}.

\subsection{Proof of \tref{;tmain}}\label{;ssptmain}
In this subsection, we prove \tref{;tmain}.

\begin{proof}[Proof of \tref{;tmain}]
First observe that, for every $\phi \in C([0,t];\R )$ and for every 
$\xi ,z\in \R $, 
\begin{align}\label{;obs1}
 \tr _{\psi _{t}-\h (\psi )}(\psi )\big| _{\psi =\tr _{\phi _{t}-\xi }(\phi )}
 =\tr _{\phi _{t}-\h _{\phi }(\xi )}(\phi ),
\end{align}
and 
\begin{align}\label{;obs2}
 \h _{\tr _{z}(\phi )}(\xi )=\h _{\phi }(\xi ).
\end{align}
As for \eqref{;obs1}, by the fact that $\tr _{\phi _{t}-\xi }(\phi )(t)=\xi $ 
as seen in \eqref{;twrel}, and by the definition of the function $\h _{\phi }$, 
the left-hand side of \eqref{;obs1} is equal to 
\begin{align*}
 \tr _{\xi -\h _{\phi }(\xi )}\bigl( 
 \tr _{\phi _{t}-\xi }(\phi )
 \bigr) ,
\end{align*}
which coincides with the right-hand side by \lref{;lttrans}\thetag{iv}. 
In view of the definition of $\h _{\phi }$, the latter observation~\eqref{;obs2} is 
verified in the same way, by noting that 
\begin{align*}
 \tr _{\tr _{z}(\phi )(t)-\xi }\bigl( \tr _{z}(\phi )\bigr) 
 &=\tr _{\phi _{t}-z-\xi }\bigl( \tr _{z}(\phi )\bigr) \\
 &=\tr _{\phi _{t}-\xi }(\phi ),
\end{align*}
which may also be seen as a consequence of \lref{;lttrans}\thetag{iii} and 
the fact that, as \eqref{;iipx} indicates, the dependence of $\tr _{\phi _{t}-\xi }(\phi )$ 
on $\phi $ is through $Z(\phi )$. Then, by the above two observations and 
\pref{;pdisint}, the left-hand side of \eqref{;qtmain} is written as 
\begin{align*}
 \ex \!\left[ 
 \int _{\R }\frac{d\xi }{2K_{0}\bigl( 
 1/Z_{t}(\phi )
 \bigr) }\,F\bigl( 
 \tr _{\phi _{t}-\h _{\phi }(\xi )}(\phi ),\tr _{\phi _{t}-\xi }(\phi )
 \bigr) \exp \left\{ 
 -\frac{\cosh \h _{\phi }(\xi )}{Z_{t}(\phi )}
 \right\} \left| 
 \h _{\phi }'(\xi )
 \right| \bigg| _{\phi =B}
 \right] ,
\end{align*}
in which, by changing the variables with $\eta =\h _{\phi }(\xi )$, 
the integral with respect to $\xi $ is equal to 
\begin{align*}
 \int _{\h _{\phi }(\R )}\frac{d\eta }{2K_{0}\bigl( 
 1/Z_{t}(\phi )
 \bigr) }\,F\!\left(  
 \tr _{\phi _{t}-\eta }(\phi ),\tr _{\phi _{t}-\h _{\phi }^{-1}(\eta )}(\phi )
 \right) \exp \left\{ 
 -\frac{\cosh \eta }{Z_{t}(\phi )}
 \right\} .
\end{align*}
Therefore, by using \pref{;pdisint} again, the above expectation coincides with 
the right-hand side of \eqref{;qtmain}.
\end{proof}

\begin{rem}\label{;rinvol1}
 \tref{;tmain} indicates that, for any $\phi \in C([0,t];\R )$ such that 
 $\phi _{t}\in \h _{\phi }(\R )$, setting $\psi \in C([0,t];\R )$ by 
 \begin{align*}
  \psi _{s}:=\tr _{\phi _{t}-\h _{\phi }^{-1}(\phi _{t})}(\phi )(s),
  \quad 0\le s\le t,
 \end{align*}
 we have 
 \begin{align}\label{;latter}
  \tr _{\psi _{t}-\h (\psi )}(\psi )=\phi ,
 \end{align}
 which is indeed the case as seen below. First note that 
 $\psi _{t}=\h _{\phi }^{-1}(\phi _{t})$ by \lref{;lttrans}\thetag{i}. 
 Moreover, by the definition of $\h _{\phi }$, 
 \begin{align*}
  \h (\psi )=\h _{\phi }\bigl( \h _{\phi }^{-1}(\phi _{t})\bigr) =\phi _{t}.
 \end{align*}
 Therefore we have 
 \begin{align*}
  \tr _{\psi _{t}-\h (\psi )}(\psi )
  =\tr _{\h _{\phi }^{-1}(\phi _{t})-\phi _{t}}\bigl( 
  \tr _{\phi _{t}-\h _{\phi }^{-1}(\phi _{t})}(\phi )
  \bigr) ,
 \end{align*}
 which is equal to $\phi $ by \lref{;lttrans}\thetag{iv}.
\end{rem}

\subsection{Proof of \cref{;cmain}}\label{;sspcmain}

We apply \tref{;tmain} to the function 
$\h (\phi )=h_{\La }(\phi _{t},\phi ),\,\phi \in C([0,t];\R )$, 
with $h_{\La }$ defined through \eqref{;dhal}. 
Notice that, for every $\phi \in C([0,t];\R )$, the 
associated function $\h _{\phi }:\R \to \R $ is given by 
\begin{align*}
 \h _{\phi }(\xi )=h_{\La }(\xi ,\phi ),\quad \xi \in \R .
\end{align*} 
Indeed, by the definition of $\h _{\phi }$ and \lref{;lttrans}\thetag{i}, 
\begin{align*}
 \h _{\phi }(\xi )&=h_{\La }\bigl( \tr _{\phi _{t}-\xi }(\phi )(t),\tr _{\phi _{t}-\xi }(\phi )\bigr) \\
 &=h_{\La }\bigl( \xi ,\tr _{\phi _{t}-\xi }(\phi )\bigr) ,
\end{align*}
of which the integrand in the defining relation~\eqref{;dhal} is the same as that of 
$h_{\La }(\xi ,\phi )$. Namely, it holds that, by properties~\thetag{i}, \thetag{iii} and \thetag{iv} 
of \lref{;lttrans}, 
\begin{align*}
 &\La \left( \tr _{\tr _{\phi _{t}-\xi }(\phi )(t)-x}\bigl( \tr _{\phi _{t}-\xi }(\phi )\bigr) \right) \exp \left\{ 
 -\frac{\cosh x}{Z_{t}\bigl( \tr _{\phi _{t}-\xi }(\phi )\bigr) }
 \right\} \\
 &=\La (\tr _{\phi _{t}-x}(\phi ))\exp \left\{ 
  -\frac{\cosh x}{Z_{t}(\phi )}
  \right\} 
\end{align*}
for any $x\in \R $. By the continuity and positivity of $\La $, together with 
the observation given just below the statement of \cref{;cmain}, it is clear that 
$\h _{\phi }$ fulfills the assumption of \tref{;tmain} and that $\h _{\phi }(\R )=\R $. 
Moreover, because relation~\eqref{;dhal} is rewritten as 
\begin{align*}
  \int _{-\infty }^{\xi }
  dx\,\La (\tr _{\phi _{t}-x}(\phi ))\exp \left\{ 
  -\frac{\cosh x}{Z_{t}(\phi )}
  \right\}  
  =\int _{h_{\La }(\xi ,\phi )}^{\infty }dx\,\La (\tr _{\phi _{t}-x}(\phi ))
  \exp \left\{ 
  -\frac{\cosh x}{Z_{t}(\phi )}
  \right\}  
\end{align*}
for any $\xi \in \R $, we have 
\begin{align}\label{;invhal}
 h_{\La }^{-1}(\xi ,\phi )=h_{\La }(\xi ,\phi ),\quad \xi \in \R ,
\end{align}
where $h_{\La }^{-1}$ denotes the inverse function of $h_{\La }$ 
in the first variable.

\begin{proof}[Proof of \cref{;cmain}]
Taking $\h (\phi )=h_{\La }(\phi _{t},\phi ),\,\phi \in C([0,t];\R )$, in \tref{;tmain}, 
we replace $F$ by a function of the form 
\begin{align*}
 F(\phi ^{1},\phi ^{2})\La (\phi ^{1}),\quad 
 (\phi ^{1},\phi ^{2})\in C([0,t];\R ^{2}),
\end{align*}
where $F$ is again a nonnegative measurable function on 
$C([0,t];\R ^{2})$. Then, with the notation $\ct _{\La }$ in 
\eqref{;dctal},  
the left-hand side of \eqref{;qtmain} turns into 
\begin{align*}
 \ex \!\left[ 
 F\bigl( \ct _{\La }(B),B\bigr) \La (\ct _{\La }(B))\exp \left\{ 
 -\frac{\cosh h_{\La }(B_{t},B)}{Z_{t}}+\frac{\cosh B_{t}}{Z_{t}}
 \right\} |h_{\La }'(B_{t},B)|
 \right] ,
\end{align*}
where the derivative of $h_{\La }$ is taken with respect to the 
first variable. The above expectation is equal to 
$
\ex \!\left[ 
F\bigl( \ct _{\La }(B),B\bigr) \La (B)
\right] 
$ since relation~\eqref{;dhal} entails that 
\begin{align*}
 &\La \bigl( 
 \tr _{\phi _{t}-h_{\La }(\xi ,\phi )}(\phi )
 \bigr) 
 \exp \left\{ 
 -\frac{\cosh h_{\La }(\xi ,\phi )}{Z_{t}(\phi )}
 \right\} h_{\La }'(\xi ,\phi )\\
 &=-\La \bigl( 
 \tr _{\phi _{t}-\xi }(\phi )
 \bigr) \exp \left\{ 
 -\frac{\cosh \xi }{Z_{t}(\phi )}
 \right\} 
\end{align*}
for every $\xi \in \R $ and $\phi \in C([0,t];\R )$, and hence, 
by inserting $\xi =B_{t}$ and $\phi =B$, 
\begin{align*}
 \La (\ct _{\La }(B))\exp \left\{ 
 -\frac{\cosh h_{\La }(B_{t},B)}{Z_{t}}
 \right\} h_{\La }'(B_{t},B)
 =-\La (B)\exp \left( 
 -\frac{\cosh B_{t}}{Z_{t}}
 \right) .
\end{align*}
Here we used the fact that $\tr _{0}=\id $ in the right-hand side of 
the last equality. On the other hand, with the above replacement of $F$ and 
in view of \eqref{;invhal}, the right-hand side of \eqref{;qtmain} becomes 
$
\ex \!\left[ 
F\bigl( B,\ct _{\La }(B)\bigr) \La (B)
\right] 
$, 
verifying the claim.
\end{proof}

\begin{rem}\label{;rinvol2}
By virtue of \eqref{;invhal}, relation~\eqref{;latter} 
reveals that $\ct _{\La }$ is an involution.
\end{rem}

\section{Examples}\label{;se}
In this section, we apply \tref{;tmain} and \cref{;cmain} to 
provide some examples. In what follows, $t>0$ is fixed 
as above and, as in the proof of 
\cref{;cmain}, the symbol $F$ refers to a generic 
nonnegative measurable function on $C([0,t];\R ^{2})$ which 
may differ in different contexts.

\subsection{Examples of \tref{;tmain}}\label{;sset}
In all of the examples below, we consider a specific case that 
$\h :C([0,t];\R )\to \R $ is of the form 
\begin{align*}
 \h (\phi )=k\bigl( \phi _{t},Z_{t}(\phi )\bigr) ,\quad \phi \in C([0,t];\R ),
\end{align*}
where $k\equiv k(t,\,\cdot\,,\cdot \,)$ is a measurable function on $\R \times (0,\infty )$ 
such that, for every $\zeta >0$, the function 
\begin{align*}
 \R \ni \xi \mapsto k(\xi ,\zeta )
\end{align*}
is of class $C^{1}$ and strictly monotone. Note that, in view of 
\eqref{;twrel}, we have $\h _{\phi }(\xi )=k\bigl( \xi ,Z_{t}(\phi )\bigr) $ 
for every $\phi \in C([0,t];\R )$ and $\xi \in \R $.
We start with restating \eqref{;qtmain} under the above setting 
assuming that 
\begin{align*}
 |k'(\xi ,\zeta )|>0 \quad \text{for all }
 (\xi ,\zeta )\in \R \times (0,\infty ).
\end{align*}
Here and in what follows, the derivative, as well as the inverse, is taken with 
respect to the first variable. In \eqref{;qtmain}, we replace $F$ by a function of the form 
\begin{align*}
 F(\phi ^{1},\phi ^{2})
 \exp \left\{ 
 \frac{\cosh k\bigl( \phi ^{2}_{t},Z_{t}(\phi ^{2})\bigr) }{Z_{t}(\phi ^{2})}
 -\frac{\cosh \phi ^{2}_{t}}{Z_{t}(\phi ^{2})}
 \right\} \frac{1}{
 \bigl| k'\bigl( 
 \phi ^{2}_{t},Z_{t}(\phi ^{2})\bigr) \bigr| 
 }, && (\phi ^{1},\phi ^{2})\in C([0,t];\R ^{2}).
\end{align*}
Then, since, denoting $\phi ^{2}=\tr _{B_{t}-k^{-1}(B_{t},Z_{t})}(B)$, 
we have 
\begin{align}\label{;ex41q1}
 \phi ^{2}_{t}=k^{-1}(B_{t},Z_{t}) && 
 \text{and} && 
 Z_{t}(\phi ^{2})=Z_{t}
\end{align}
by \thetag{i} and \thetag{iii} of \lref{;lttrans}, relation~\eqref{;qtmain} 
turns into 
\begin{equation}\label{;qtmaind}
  \begin{split}
  &\ex \!\left[ 
  F\bigl( 
  \tr _{B_{t}-k(B_{t},Z_{t})}(B),B
  \bigr) 
  \right] \\
  &=\ex \biggl[ 
  F\bigl( 
  B,\tr _{B_{t}-k^{-1}(B_{t},Z_{t})}(B)
  \bigr) \exp \left\{ 
  \frac{\cosh B_{t}}{Z_{t}}-\frac{\cosh k^{-1}(B_{t},Z_{t})}{Z_{t}}
  \right\} \\
  &\hspace{55mm}\times \bigl| (k^{-1})'(B_{t},Z_{t})\bigr| ;\,
  (B_{t},Z_{t})\in D_{k}
  \biggr] ,
 \end{split}
\end{equation}
where a measurable set $D_{k}\subset \R \times (0,\infty )$ is defined by 
\begin{align*}
 D_{k}:=\left\{ 
 (\xi ,\zeta )\in \R \times (0,\infty );\,\xi \in k(\R ,\zeta )
 \right\} ,
\end{align*}
with $k(\R ,\zeta )$ the image of $\R $ under $k(\,\cdot \,,\zeta )$ 
for every $\zeta >0$.

\begin{exm}\label{;ex411}
Given $\al \ge 0$, let $k(\xi ,\zeta )=\xi -\log (1+\al \zeta )$. Then 
$D_{k}=\R \times (0,\infty )$ and 
$k^{-1}(\xi ,\zeta )=\xi +\log (1+\al \zeta )$. Noting that 
\begin{align*}
 \frac{\cosh \xi }{\zeta }-\frac{\cosh k^{-1}(\xi ,\zeta )}{\zeta }
 =\frac{\al }{2}\left( \frac{e^{-\xi }}{1+\al \zeta }-e^{\xi }\right) 
\end{align*}
for every $\xi \in \R $ and $\zeta >0$,  
we have, from \eqref{;qtmaind}, 
\begin{align*}
 &\ex \!\left[ 
 F\bigl( 
 \tr _{\log (1+\al Z_{t})}(B),B
 \bigr) 
 \right] \\
 &=\ex \!\left[ F\bigl( 
 B,\tr _{-\log (1+\al Z_{t})}(B)
 \bigr) \exp \left\{ 
 \frac{\al }{2}\left( \frac{e^{-B_{t}}}{1+\al Z_{t}}-e^{B_{t}}\right) 
 \right\} 
 \right] .
\end{align*}
For every fixed $\mu \in \R $, we further replace $F$ by 
\begin{align}\label{;fcm}
 F(\phi ^{1},\phi ^{2})e^{\mu \phi ^{2}_{t}-\mu ^{2}t/2},\quad 
 (\phi ^{1},\phi ^{2})\in C([0,t];\R ^{2}).
\end{align}
Then, noting that, as to the right-hand side, 
$
\exp \bigl\{ \tr _{-\log (1+\al Z_{t})}(B)(t)\bigr\}
=e^{B_{t}}(1+\al Z_{t})  
$ in view of the former relation in \eqref{;ex41q1}, 
we obtain, by the Cameron--Martin formula,
\begin{align*}
 &\ex \!\left[ 
 F\bigl( 
 \tr _{\log \{ 1+\al \dz{\mu }_{t}\} }(\db{\mu }),\db{\mu }
 \bigr) 
 \right] \\
 &=\ex \!\left[ F\bigl( 
 \db{\mu },\tr _{-\log \{ 1+\al \dz{\mu }_{t}\} }(\db{\mu })
 \bigr) \exp \left\{ 
 \frac{\al }{2}\left( \frac{e^{-\db{\mu }_{t}}}{1+\al \dz{\mu }_{t}}
 -e^{\db{\mu }_{t}}
 \right) 
 \right\} \bigl\{ 1+\al \dz{\mu }_{t}\bigr\} ^{\mu }
 \right] .
\end{align*}
When $F$ is independent of the second variable $\phi ^{2}$, the above relation is 
Theorem~1.1 in \cite{dmy} since, by the definition \eqref{;ttrans} of $\{ \tr _{z}\} _{z\in \R }$, 
\begin{align*}
 \tr _{\log \{ 1+\al \dz{\mu }_{t}\} }(\db{\mu })(s)
 =\db{\mu }_{s}-\log \left\{ 
 1+\al e^{-\db{\mu }_{t}}\!\da{\mu }_{s}
 \right\} ,\quad 0\le s\le t,
\end{align*}
where we have used the relation 
$\dz{\mu }_{t}=e^{-\db{\mu }_{t}}\!\da{\mu }_{t}$ 
by the definition~\eqref{;defz} of the transformation $Z$.
\end{exm}

\begin{exm}\label{;ex4112}
Given $\al \ge 0$, let $k(\xi ,\zeta )=-\log (e^{-\xi }+\al \zeta )$. 
Then  
\begin{align*}
 D_{k}=\left\{ 
 (\xi ,\zeta )\in \R \times (0,\infty );\,1/(e^{\xi }\zeta )>\al 
 \right\} ,
\end{align*}
on which we have $k^{-1}(\xi ,\zeta )=-\log (e^{-\xi }-\al \zeta )$ and 
\begin{align*}
 \frac{\cosh \xi }{\zeta }-\frac{\cosh k^{-1}(\xi ,\zeta )}{\zeta }
 =\frac{\al }{2}\left( 
 1-\frac{e^{2\xi }}{1-\al e^{\xi }\zeta }
 \right) ,
\end{align*}
as well as $(k^{-1})'(\xi ,\zeta )=1/(1-\al e^{\xi }\zeta )$.
Therefore, by \eqref{;qtmaind}, 
\begin{equation}\label{;ex412q1}
\begin{split}
 &\ex \!\left[ 
 F\bigl( 
 \tr _{\log (1+\al A_{t})}(B),B
 \bigr) \right] \\
 &=\ex \!\left[ 
 F\bigl( 
 B,\tr _{\log (1-\al A_{t})}(B)
 \bigr) \exp \left\{ \frac{\al }{2}\left( 
 1-\frac{e^{2B_{t}}}{1-\al A_{t}}
 \right) \right\} \frac{1}{1-\al A_{t}}
 ;\, \frac{1}{A_{t}}>\al 
 \right] ,
\end{split}
\end{equation}
where we have used the relation $A_{t}=e^{B_{t}}Z_{t}$. 
For every fixed $\mu \in \R $, we replace $F$ by a 
function of the form~\eqref{;fcm} to deduce further that, 
by the Cameron--Martin formula, 
\begin{align*}
 &\ex \!\left[ 
 F\bigl( 
 \tr _{\log \{ 1+\al \da{\mu }_{t}\} }(\db{\mu }),\db{\mu }
 \bigr) \right] \\
 &=\ex \Biggl[ 
 F\bigl( 
 \db{\mu },\tr _{\log \{ 1-\al \da{\mu }_{t}\} }(\db{\mu })
 \bigr) \exp \left\{ \frac{\al }{2}\left( 
 1-\frac{e^{2\db{\mu }_{t}}}{1-\al \da{\mu }_{t}}
 \right) \right\} \\
 &\hspace{56mm}\times \frac{1}{\bigl\{ 1-\al \da{\mu }_{t}\bigr\} ^{1+\mu }}
 ;\, \frac{1}{\da{\mu }_{t}}>\al 
 \Biggr] ,
\end{align*}
noting that 
$
\exp \bigl\{ \tr _{\log (1-\al A_{t})}(B)(t)\bigr\} 
=e^{B_{t}}/(1-\al A_{t})
$ 
by \lref{;lttrans}\thetag{i} as to the right-hand side. 
Since, by the definition \eqref{;ttrans} of $\{ \tr _{z}\} _{z\in \R }$, 
the transformation of the form 
\begin{align}\label{;nonantic}
 \tr _{\log \{ 1+\al A_{t}(\phi )\} }(\phi ),\quad 
 \phi \in C([0,t];\R ),
\end{align}
is expressed as 
\begin{align*}
 \phi _{s}-\log \left\{ 
 1+\al A_{s}(\phi )
 \right\} ,\quad 0\le s\le t,
\end{align*}
the last displayed relation extends 
\cite[Theorem~1.5]{dmy} particularly to the case that 
$\mu $ is allowed to take negative values.
\end{exm}

\begin{rem}\label{;rnonantic}
Although the transformation~\eqref{;nonantic} is non-anticipative, 
it is clear that relation~\eqref{;ex412q1} is not the one that follows from 
Girsanov's formula, for which we also refer to \cite[Remark~1.1]{dmy}.
\end{rem}

\begin{exm}\label{;ex413}
Given $z\in \R $, let $k(\xi ,\zeta )=\xi -z$. Then  
$D_{k}=\R \times (0,\infty )$ and  
$k^{-1}(\xi ,\zeta )=\xi +z$, whence,   
by \eqref{;qtmaind}, 
\begin{align*}
 \ex \!\left[ 
 F\bigl( \tr _{z}(B),B\bigr) 
 \right] 
 =\ex \!\left[ 
 F\bigl( B,\tr _{-z}(B)\bigr) 
 \exp \left\{ 
 \frac{\cosh B_{t}}{Z_{t}}-\frac{\cosh (B_{t}+z)}{Z_{t}}
 \right\} 
 \right] ,
\end{align*}
which recovers \cite[Theorem~1.2]{har22}. The above 
relation is consistent with the property 
$\tr _{z}\circ \tr _{-z}=\id $ in \lref{;lttrans}\thetag{iv}.
\end{exm}

We use the notation in \cite{har22a} to denote 
\begin{align}\label{;dct}
 \ct (\phi )(s):=\tr _{2\phi _{t}}(\phi )(s),\quad 
 0\le s\le t,
\end{align}
for $\phi \in C([0,t];\R )$.

\begin{exm}\label{;ex414}
In this example, we let $k(\xi ,\zeta )=-\xi $ in \eqref{;qtmaind} to see that 
\begin{align}\label{;invbex}
 \ex \!\left[ 
 F\bigl( 
 \tr _{2B_{t}}(B),B
 \bigr) 
 \right] 
 =\ex \!\left[ 
 F\bigl( 
 B,\tr _{2B_{t}}(B)
 \bigr) 
 \right] ,
\end{align}
that is, with the notation recalled above, we have 
\begin{align}\label{;invb}
 \left\{ 
 \bigl( 
 \ct (B)(s),\,B_{s}
 \bigr) 
 \right\} _{0\le s\le t}
 \eqd 
 \left\{ 
 \bigl( 
 B_{s},\,\ct (B)(s)
 \bigr) 
 \right\} _{0\le s\le t}, 
\end{align}
which is \cite[Theorem~1.1]{har22a} or, more precisely, 
\cite[Corollary~1.1]{har22a} with $\mu =0$ therein. 
In particular, the Wiener measure on $C([0,t];\R )$ is 
invariant under $\ct $. A generalization of \eqref{;invb} 
to the case of Brownian motion with drift or other diffusion 
processes is given in \ssref{;ssec} as an application of 
\cref{;cmain}. For properties of $\ct $ such as 
$\ct \circ \ct =\id $ and the compatibility with the 
time-reversal operator $R$ defined in \eqref{;trev} 
that follows from \eqref{;comptz}, we refer to 
\cite[Proposition~2.1]{har22a}.
\end{exm}

\begin{exm}\label{;ex415}
Given $x\ge 0$, we consider the following two cases: 
\begin{align*}
 \thetag{i}~k(\xi ,\zeta )=\log (e^{-\xi }+2x\zeta ); 
 &&   
 \thetag{ii}~k(\xi ,\zeta )=-\log (e^{\xi }+2x\zeta ).
\end{align*}

\thetag{i} In this case, 
\begin{align*}
 D_{k}=\left\{ 
 (\xi ,\zeta )\in \R \times (0,\infty );\,e^{\xi }/(2\zeta )>x
 \right\} ,
\end{align*}
on which we have $k^{-1}(\xi ,\zeta )=-\log (e^{\xi }-2x\zeta )$, 
\begin{align*}
 \frac{\cosh \xi }{\zeta }-\frac{\cosh k^{-1}(\xi ,\zeta )}{\zeta }
 =x-\frac{x}{e^{2\xi }-2xe^{\xi }\zeta },
\end{align*}
and $(k^{-1})'(\xi ,\zeta )=-e^{\xi }/(e^{\xi }-2x\zeta )$.
Therefore, by \eqref{;qtmaind} and by recalling the relation 
$A_{t}=e^{B_{t}}Z_{t}$,
\begin{equation}\label{;ex415q1}
\begin{split}
 &\ex \!\left[ 
 F\bigl( 
 \tr _{\log \{ e^{2B_{t}}/(1+2xA_{t})\} }(B),B \bigr) 
 \right] \\
 &=\ex \!\left[ 
 F\bigl( 
 B,\tr _{\log (e^{2B_{t}}-2xA_{t})}(B)\bigr) 
 \exp \left( 
 x-\frac{x}{e^{2B_{t}}-2xA_{t}}
 \right) \frac{e^{2B_{t}}}{e^{2B_{t}}-2xA_{t}};\,
 \frac{e^{2B_{t}}}{2A_{t}}>x
 \right] ,
\end{split}
\end{equation}
which is the former relation in \cite[Proposition~5.3]{har22a}.  

\thetag{ii} In the case $k(\xi ,\zeta )=-\log (e^{\xi }+2x\zeta )$, 
\begin{align*}
 D_{k}=\left\{ 
 (\xi ,\zeta )\in \R \times (0,\infty );\,1/(2e^{\xi }\zeta )>x
 \right\} ,
\end{align*}
on which we have $k^{-1}(\xi ,\zeta )=\log (e^{-\xi }-2x\zeta )$ and 
\begin{align*}
 \frac{\cosh \xi }{\zeta }-\frac{\cosh k(\xi ,\zeta )}{\zeta }
 =x-\frac{xe^{2\xi }}{1-2xe^{\xi }\zeta },
\end{align*}
as well as $(k^{-1})'(\xi ,\zeta )=-1/(1-2xe^{\xi }\zeta )$.
Therefore, by \eqref{;qtmaind},
\begin{equation}\label{;ex415q2}
\begin{split}
 &\ex \!\left[ 
 F\bigl( 
 \tr _{\log (e^{2B_{t}}+2xA_{t})}(B),B\bigr) 
 \right] \\
 &=\ex \!\left[ 
 F\bigl( 
 B,\tr _{\log \{ e^{2B_{t}}/(1-2xA_{t})\} }(B)
 \bigr) \exp \left( 
 x-\frac{xe^{2B_{t}}}{1-2xA_{t}}
 \right) \frac{1}{1-2xA_{t}};\,
 \frac{1}{2A_{t}}>x
 \right] ,
\end{split}
\end{equation}
which is the latter relation in \cite[Proposition~5.3]{har22a}.
\end{exm}

\begin{rem}\label{;requiv}
With $\al =2x$ in \eqref{;ex412q1}, the above three 
relations~\eqref{;ex412q1}, \eqref{;ex415q1} and \eqref{;ex415q2} are 
equivalent and related via the identity~\eqref{;invb} in law. For instance, 
if, in \eqref{;ex412q1}, we replace $F$ by a function of the form 
\begin{align*}
 F\bigl( 
 \phi ^{1},\ct (\phi ^{2})
 \bigr) ,\quad (\phi ^{1},\phi ^{2})\in C([0,t];\R ^{2}),
\end{align*}
then the left-hand side turns into 
\begin{align*}
 \ex \!\left[ 
 F\bigl( 
 \tr _{\log (1+\al A_{t})}(B),\ct (B)
 \bigr) 
 \right] =\ex \!\left[ 
 F\bigl( 
 \tr _{\log \{ 1+\al A_{t}(\ct (B))\} }(\ct (B)),B
 \bigr) 
 \right] 
\end{align*}
owing to \eqref{;invb}, which agrees with the left-hand side of 
\eqref{;ex415q2} because, by the 
definition~\eqref{;dct} of $\ct $, 
\begin{align*}
 \tr _{\log \{ 1+\al A_{t}(\ct (B))\} }(\ct (B))
 &=\tr _{\log (1+\al e^{-2B_{t}}A_{t})}\bigl( \tr _{2B_{t}}(B)\bigr) \\
 &=\tr _{\log (1+\al e^{-2B_{t}}A_{t})+2B_{t}}(B),
\end{align*}
thanks to \thetag{ii} and \thetag{iv} of \lref{;lttrans}. 
On the other hand, as for the right-hand side, note that, 
by the definition of $\ct $ and properties~\thetag{i} and 
\thetag{iv} of \lref{;lttrans}, 
\begin{align*}
 \ct \bigl( 
 \tr _{\log (1-\al A_{t})}(B)
 \bigr) &=\tr _{2B_{t}-2\log (1-\al A_{t})}
 \bigl( 
 \tr _{\log (1-\al A_{t})}(B)
 \bigr) \\
 &=\tr _{2B_{t}-\log (1-\al A_{t})}(B)
\end{align*}
on the event that $A_{t}<1/\al $, and hence, 
with the above replacement of $F$ and $\al =2x$, 
the right-hand side of \eqref{;ex412q1} agrees with 
that of \eqref{;ex415q2}. Other implications between 
these three relations may be verified in a similar manner 
(cf.\ \cite[Remarks~5.3 and 5.4]{har22a}).
\end{rem}

\subsection{Examples of \cref{;cmain}}\label{;ssec}

In this subsection, we explore several examples of \cref{;cmain}. 
We mainly focus on a specific case that 
$\La :C([0,t];\R )\to (0,\infty )$ is of the form 
\begin{align*}
 \La (\phi )=\la \bigl( \phi _{t},Z_{t}(\phi )\bigr) ,\quad \phi \in C([0,t];\R ),
\end{align*}
with $\la \equiv \la (t,\,\cdot\,,\cdot \,):\R \times (0,\infty )\to (0,\infty )$ 
a continuous function satisfying 
\begin{align*}
 \int _{\R }d\xi \,\la (\xi ,\zeta )
 \exp \left( 
 -\frac{\cosh \xi }{\zeta }
 \right) <\infty \quad \text{for all $\zeta >0$}.
\end{align*}
Then, because of the fact that, for every $\xi \in \R $ and 
$\phi \in C([0,t];\R )$, 
\begin{align*}
 \La \bigl( \tr _{\phi _{t}-\xi }(\phi )\bigr) 
 =\la \bigl( \xi ,Z_{t}(\phi )\bigr) 
\end{align*}
in view of \eqref{;twrel}, \cref{;cmain} is restated in such a way that, for every nonnegative measurable 
function $F$ on $C([0,t];\R ^{2})$,  
\begin{align}\label{;qcmaind}
 \ex \!\left[ 
 F\bigl( 
 \tr _{B_{t}-h_{\la }(B_{t},Z_{t})}(B),B
 \bigr) \la (B_{t},Z_{t})
 \right] 
 =\ex \!\left[ 
 F\bigl( B,\tr _{B_{t}-h_{\la }(B_{t},Z_{t})}(B)\bigr) \la (B_{t},Z_{t})
 \right] ,
\end{align}
where, with slight abuse of notation, $h_{\la }$ is defined through
\begin{equation}\label{;dhald}
\begin{split}
 \int _{-\infty }^{h_{\la }(\xi ,\zeta )}
 dx\,\la (x,\zeta )\exp \left( 
 -\frac{\cosh x}{\zeta }
 \right) =\int _{\xi }^{\infty }dx\,\la (x,\zeta )
 \exp \left( 
 -\frac{\cosh x}{\zeta }
 \right) 
\end{split} 
\end{equation}
for $\xi \in \R $ and $\zeta >0$. We see from \eqref{;invhal} that  
\begin{align}\label{;invhald}
 h_{\la }^{-1}(\xi ,\zeta )=h_{\la }(\xi ,\zeta )
\end{align}
for all $\xi \in \R $ and $\zeta >0$.

\begin{exm}\label{;ex421}
We consider the case where $\la (\,\cdot \,,\zeta )$ is an even 
function for every $\zeta >0$, in which case we have 
\begin{align*}
 h_{\la }(\xi ,\zeta )=-\xi ,\quad \xi \in \R ,
\end{align*}
since relation~\eqref{;dhald} is rewritten as 
\begin{align*}
 \int _{-\infty }^{h_{\la }(\xi ,\zeta )}
 dx\,\la (x,\zeta )\exp \left( 
 -\frac{\cosh x}{\zeta }
 \right) 
 &=\int _{-\infty }^{-\xi }dx\,\la (-x,\zeta )
 \exp \left( 
 -\frac{\cosh x}{\zeta }
 \right) \\
 &=\int _{-\infty }^{-\xi }
 dx\,\la (x,\zeta )\exp \left( 
 -\frac{\cosh x}{\zeta }
 \right) .
\end{align*}
Therefore, with the notation in \eqref{;dct}, relation~\eqref{;qcmaind} 
becomes 
\begin{align}\label{;sym}
 \ex \!\left[ 
 F\bigl( 
 \ct (B),B
 \bigr) \la (B_{t},Z_{t})
 \right] 
 = \ex \!\left[ 
 F\bigl( 
 B,\ct (B)
 \bigr) \la (B_{t},Z_{t})
 \right] .
\end{align}

\thetag{1} On one hand, the last relation indicates the identity~\eqref{;invb} in law, 
which is seen not only by taking $\la \equiv 1$ but also by replacing $F$ by 
a function of the form 
\begin{align*}
 F(\phi ^{1},\phi ^{2})/\la \bigl( \phi ^{2}_{t},Z_{t}(\phi ^{2})\bigr) ,
 \quad (\phi ^{1},\phi ^{2})\in C([0,t];\R );
\end{align*}
indeed, with the above replacement, relation~\eqref{;sym} 
turns into 
\begin{align*}
 \ex \!\left[ 
 F\bigl( 
 \ct (B),B
 \bigr) 
 \right] 
 &= \ex \!\left[ 
 F\bigl( 
 B,\ct (B)
 \bigr) \frac{\la (B_{t},Z_{t})}{
 \la (-B_{t},Z_{t})
 }
 \right] \\
 &=\ex \!\left[ 
 F\bigl( 
 B,\ct (B)
 \bigr) 
 \right] ,
\end{align*}
where, for the first line, we have used the fact that, 
for every $\phi \in C([0,t];\R )$,
\begin{align*}
 \ct (\phi )(t)=-\phi _{t} && \text{and} && 
 Z_{t}\bigl( \ct (\phi )\bigr) =Z_{t}(\phi )
\end{align*}
by the definition~\eqref{;dct} of $\ct $ and properties~\thetag{i} and 
\thetag{iii} of \lref{;lttrans}.

\thetag{2} On the other hand, one may also deduce from \eqref{;sym} 
the distributional invariance under the transformation $\ct $ of processes 
that differ from Brownian motion. As an illustration, we consider the 
following example: given $\mu \in \R $, let 
\begin{align*}
 \la (\xi ,\zeta )=\cosh (\mu \xi )e^{-\mu ^{2}t/2},\quad \xi \in \R ,
\end{align*}
for every $\zeta >0$. Then, by noting that the process 
\begin{align*}
 \cosh (\mu B_{s})e^{-\mu ^{2}s/2},\quad s\ge 0,
\end{align*}
is a martingale with initial value $1$, Girsanov's formula 
entails that the law of the solution 
$X^{(\mu )}=\bigl\{ X^{(\mu )}_{s}\bigr\} _{0\le s\le t}$ to 
the stochastic differential equation (SDE)
\begin{align*}
 dX_{s}=dB_{s}+\mu \tanh (\mu X_{s})\,ds,\quad X_{0}=0,
\end{align*}
is also invariant under $\ct $; more precisely, 
\begin{align*}
 \left( 
 \ct \bigl( X^{(\mu )}\bigr) ,\,X^{(\mu )}
 \right) 
 \eqd 
 \left( 
 X^{(\mu )},\,\ct \bigl( X^{(\mu )}\bigr) 
 \right) .
\end{align*}
The case $\mu =0$ agrees with \eqref{;invb}.
\end{exm}

\begin{exm}\label{;ex422}
Given $\mu \in \R $, let 
\begin{align}\label{;wcm}
 \la (\xi ,\zeta )=\exp \left( 
 \mu \xi -\frac{\mu ^{2}}{2}t
 \right) , \quad \xi \in \R ,
\end{align}
for every $\zeta >0$. Then, with the notation in \eqref{;dctal}, 
\cref{;cmain} entails that, by the Cameron--Martin formula, 
\begin{align}\label{;invbd}
 \left\{ 
 \bigl( 
 \ct _{\La }(\db{\mu })(s),\,\db{\mu }_{s}
 \bigr) 
 \right\} _{0\le s\le t}
 \eqd 
 \left\{ 
 \bigl( 
 \db{\mu }_{s},\,\ct _{\La }(\db{\mu })(s)
 \bigr) 
 \right\} _{0\le s\le t}, 
\end{align}
which extends \eqref{;invb} to the case of 
Brownian motion with drift, for the case $\mu =0$ 
corresponds to the case $\la \equiv 1$ in \eref{;ex421}.
\end{exm}

Recall from \cite[Corollary~1.1]{har22a} that the laws of 
Brownian motions with opposite drifts are related via 
\begin{align}\label{;opp}
 \left\{ 
 \bigl( 
 \ct (\db{-\mu })(s),\,\db{-\mu }_{s}
 \bigr) 
 \right\} _{0\le s\le t}
 \eqd 
 \left\{ 
 \bigl( 
 \db{\mu }_{s},\,\ct (\db{\mu })(s)
 \bigr) 
 \right\} _{0\le s\le t}
\end{align}
for every $\mu \in \R $, which is seen 
by replacing $F$ in \eqref{;invbex} by a function of the form 
\begin{align*}
 F(\phi ^{1},\phi ^{2})e^{\mu \phi ^{1}_{t}},\quad 
 (\phi ^{1},\phi ^{2})\in C([0,t];\R ^{2}).
\end{align*}
The above example enables us to obtain another distributional 
relationship between $\db{\mu }$ and $\db{-\mu}$ as in the 
proposition below. We denote 
by $k_{\mu }$ the function $h_{\la }$ corresponding to \eqref{;wcm}, 
namely $k_{\mu }$ is defined through 
\begin{align*}
 \int _{-\infty }^{k_{\mu }(\xi ,\zeta )}dx\,e^{\mu x}
 \exp \left( 
 -\frac{\cosh x}{\zeta }
 \right) 
 =\int _{\xi }^{\infty }dx\,e^{\mu x}
 \exp \left( 
 -\frac{\cosh x}{\zeta }
 \right) 
\end{align*}
for $\xi \in \R $ and $\zeta >0$. It is readily seen that 
\begin{align}\label{;pm}
 -k_{-\mu }(\xi ,\zeta )=k_{\mu }(-\xi ,\zeta )
\end{align}
for all $\xi \in \R $ and $\zeta >0$. We denote by 
$\cS _{\mu }$ the path transformation defined by 
\begin{align*}
 \cS _{\mu }(\phi )(s)
 :=\tr _{\phi _{t}+k_{\mu }(\phi _{t},Z_{t}(\phi ))}(\phi )(s),
 \quad 0\le s\le t,
\end{align*}
for $\phi \in C([0,t];\R )$. It then holds that 
\begin{align}\label{;comps}
 \cS _{\mu }\circ \cS _{-\mu }=\id .
\end{align}
Indeed, pick $\psi \in C([0,t];\R )$ arbitrarily and 
set $\phi =\cS _{-\mu }(\psi )$. Then we have, 
by \thetag{i} and \thetag{iii} of \lref{;lttrans}, 
\begin{align*}
 \phi _{t}+k_{\mu }(\phi _{t},Z_{t}(\phi ))
 &=-k_{-\mu }(\psi _{t},Z_{t}(\psi ))
 +k_{\mu }\bigl( 
 -k_{-\mu }(\psi _{t},Z_{t}(\psi )),Z_{t}(\psi )
 \bigr) \\
 &=-k_{-\mu }(\psi _{t},Z_{t}(\psi ))
 +k_{\mu }\bigl( 
 k_{\mu }(-\psi _{t},Z_{t}(\psi )),Z_{t}(\psi )
 \bigr) \\
 &=-k_{-\mu }(\psi _{t},Z_{t}(\psi ))-\psi _{t}, 
\end{align*}
where we have used \eqref{;pm} for the second line and 
applied \eqref{;invhald} to $k_{\mu }$ for the third. Therefore 
\begin{align*}
 \cS _{\mu }(\phi )
 &=\tr _{-k_{-\mu }(\psi _{t},Z_{t}(\psi ))-\psi _{t}}
 \bigl( 
 \tr _{\psi _{t}+k_{-\mu }(\psi _{t},Z_{t}(\psi ))}(\psi )
 \bigr) , 
\end{align*}
which, by \lref{;lttrans}\thetag{iv}, is equal to 
$\psi $ as claimed. 

\begin{prop}\label{;relopp}
For every $\mu \in \R $, we have  
\begin{align*}
 \left\{ 
 \bigl( 
 \cS _{-\mu }(\db{-\mu })(s),\,\db{-\mu }_{s}
 \bigr) 
 \right\} _{0\le s\le t}
 \eqd 
 \left\{ 
 \bigl( 
 \db{\mu }_{s},\,\cS _{\mu }(\db{\mu })(s)
 \bigr) 
 \right\} _{0\le s\le t}.
\end{align*}
\end{prop}

\begin{proof}
By \eqref{;invbd}, we have 
\begin{align*}
 \left\{ 
 \bigl( 
 \ct _{\La }(\db{\mu })(s),\,\ct (\db{\mu })(s)
 \bigr) 
 \right\} _{0\le s\le t}
 \eqd 
 \left\{ 
 \bigl( 
 \db{\mu }_{s},\,(\ct \circ \ct _{\La })(\db{\mu })(s)
 \bigr) 
 \right\} _{0\le s\le t}.
\end{align*}
By \eqref{;opp}, the left-hand side is identical in law with 
\begin{align*}
 \left\{ 
 \bigl( 
 (\ct _{\La }\circ \ct )(\db{-\mu })(s),\,\db{-\mu }_{s}
 \bigr) 
 \right\} _{0\le s\le t}.
\end{align*}
Therefore, in order to prove the proposition, it suffices to verify 
the following two relations: 
\begin{align*}
 \ct _{\La }\circ \ct =\cS _{-\mu }; && 
 \ct \circ \ct _{\La }=\cS _{\mu }.
\end{align*}
As for the former, for every $\phi \in C([0,t];\R )$, we have 
\begin{align*}
 \ct _{\La }\bigl( \ct (\phi )\bigr) 
 &=\tr _{\ct (\phi )(t)-k_{\mu }(\ct (\phi )(t),Z_{t}(\ct (\phi )))}
 \bigl( \ct (\phi )\bigr) \\
 &=\tr _{-\phi _{t}-k_{\mu }(-\phi _{t},Z_{t}(\phi ))}
 \bigl( \tr _{2\phi _{t}}(\phi )\bigr) \\
 &=\tr _{\phi _{t}+k_{-\mu }(\phi _{t},Z_{t}(\phi ))}(\phi ),
\end{align*}
which is $\cS _{-\mu }(\phi )$, where the second line 
follows from \thetag{i} and \thetag{iii} of \lref{;lttrans} 
together with the definition~\eqref{;dct} of $\ct $, and 
the third from \eqref{;pm}. 
Since $(\ct _{\La }\circ \ct )^{-1}=\ct \circ \ct _{\La }$, 
we also obtain the latter thanks to \eqref{;comps}.
\end{proof}

We return to examples of \cref{;cmain}.
\begin{exm}\label{;ex423}
Given $\mu \in \R $ and $\al >0$, let 
\begin{align*}
 \la (\xi ,\zeta )
 =\frac{K_{\mu }(\al e^{\xi })}{K_{\mu }(\al )}
 \exp \left( 
 -\frac{\al ^{2}}{2}e^{\xi }\zeta -\frac{\mu ^{2}}{2}t
 \right) ,\quad \xi \in \R ,\ \zeta >0,
\end{align*}
where $K_{\mu }$ is the modified Bessel function of 
the third kind (or the Macdonald function) of order $\mu $ 
(see, e.g., \cite[Section~5.7]{leb}). Then, with the corresponding 
$\ct _{\La }$, we have, from \eqref{;qcmaind},
\begin{equation*}
\begin{split}
 &\ex \!\left[ 
 F\bigl( 
 \ct _{\La }(B),B
 \bigr) \frac{K_{\mu }(\al e^{B_{t}})}{K_{\mu }(\al )}
 \exp \left( 
 -\frac{\al ^{2}}{2}A_{t}-\frac{\mu ^{2}}{2}t
 \right) 
 \right] \\
 &=\ex \!\left[ 
 F\bigl( 
 B,\ct _{\La }(B)
 \bigr) \frac{K_{\mu }(\al e^{B_{t}})}{K_{\mu }(\al )}
 \exp \left( 
 -\frac{\al ^{2}}{2}A_{t}-\frac{\mu ^{2}}{2}t
 \right) 
 \right] ,
\end{split}
\end{equation*}
noting $e^{B_{t}}Z_{t}=A_{t}$. Notice that the process 
\begin{align*}
 \frac{K_{\mu }(\al e^{B_{s}})}{K_{\mu }(\al )}
 \exp \left( 
 -\frac{\al ^{2}}{2}A_{s}-\frac{\mu ^{2}}{2}s
 \right) ,\quad s\ge 0,
\end{align*}
is a martingale with initial value $1$. Therefore, applying Girsanov's formula, 
we see that the law of the solution 
$X^{(\al ,\mu )}=\bigl\{ X^{(\al ,\mu )}_{s}\bigr\} _{0\le s\le t}$ 
to the SDE
\begin{align}\label{;sde}
 dX_{s}=dB_{s}+\left\{ 
 \mu -\al e^{X_{s}}\biggl( \frac{K_{\mu +1}}{K_{\mu }}\biggr) 
 \bigl( \al e^{X_{s}}\bigr) 
 \right\} ds,\quad X_{0}=0,
\end{align}
is invariant under $\ct _{\La }$; in fact, 
\begin{align*}
 \left( 
 \ct _{\La }\bigl( X^{(\al ,\mu )}\bigr) ,\,X^{(\al ,\mu )}
 \right) 
 \eqd 
 \left( 
 X^{(\al ,\mu )},\,\ct _{\La }\bigl( X^{(\al ,\mu )}\bigr) 
 \right) .
\end{align*}
The expression of the drift term in the above SDE is due to the relation 
\begin{align*}
 \frac{d}{dz}\left\{ 
 z^{-\mu }K_{\mu }(z)
 \right\} =-z^{-\mu }K_{\mu +1}(z)
\end{align*}
(see, e.g., \cite[Equation~\thetag{5.7.9}]{leb}).
\end{exm}

\begin{rem}
\thetag{1} Because of the fact that $K_{\mu }=K_{|\mu |}$ 
(see, e.g., \cite[Equation~\thetag{5.7.10}]{leb}), the law of 
$X^{(\al ,\mu )}$ is the same as that of $X^{(\al ,|\mu |)}$, and 
so is the drift term of SDE~\eqref{;sde}, which is indeed 
the case thanks to the recurrence relation 
\begin{align*}
 K_{\mu -1}(z)-K_{\mu +1}(z)=-\frac{2\mu }{z}K_{\mu }(z)
\end{align*}
(see, e.g., \cite[Equation~\thetag{5.7.9}]{leb}).

\thetag{2} We see from \cite[Theorem~1.5$'$]{myPI} that 
the infinitesimal generator of the diffusion process 
$\bigl\{ -\log \dz{\mu }_{s}\bigr\} _{s>0}$, namely 
\begin{align*}
 \db{\mu }_{s}-\log \da{\mu }_{s},\quad s>0,
\end{align*}
is the same as that of $X^{(1,\mu )}$.
\end{rem}

The next example utilizes \cref{;cmain} in full generality.

\begin{exm}\label{;ex424}
We deal with Ornstein--Uhlenbeck processes as an example; see, e.g., 
\cite[pp.~140 and 141]{bs} for their precise description. For every $\al \in \R $, 
consider the function $\La :C([0,t];\R )\to (0,\infty )$ given by 
\begin{align*}
 \La (\phi )=
 \exp \left( 
 -\frac{\al }{2}\phi _{t}^{2}+\frac{\al }{2}t-\frac{\al ^{2}}{2}\int _{0}^{t}\phi _{s}^{2}\,ds
 \right) ,\quad \phi \in C([0,t];\R ),
\end{align*}
which fulfills the assumption of \cref{;cmain}. The associated function 
$h_{\La }:\R \times C([0,t];\R )\to \R $ is defined through 
\begin{align*}
 &\int _{-\infty }^{h_{\La }(\xi ,\phi )}
 dx\,\exp \left\{ 
 -\frac{\al }{2}x^{2}-\frac{\al ^{2}}{2}\int _{0}^{t}\bigl( 
 \tr _{\phi _{t}-x}(\phi )(s)
 \bigr) ^{2}\,ds-\frac{\cosh x}{Z_{t}(\phi )}
 \right\} \\
 &=\int _{\xi }^{\infty }
 dx\,\exp \left\{ 
 -\frac{\al }{2}x^{2}-\frac{\al ^{2}}{2}\int _{0}^{t}\bigl( 
 \tr _{\phi _{t}-x}(\phi )(s)
 \bigr) ^{2}\,ds-\frac{\cosh x}{Z_{t}(\phi )}
 \right\} 
\end{align*}
for $\xi \in \R $ and $\phi \in C([0,t];\R )$.
Then, as a consequence of \cref{;cmain}, the law of the Ornstein--Uhlenbeck 
process $X=\{ X_{s}\} _{0\le s\le t}$ described by the SDE 
\begin{align*}
 dX_{s}=dB_{s}-\al X_{s}\,ds,\quad X_{0}=0,
\end{align*}
is invariant under the corresponding transformation $\ct _{\La }$, or more precisely, 
\begin{align*}
 \bigl( \ct _{\La }(X),\,X\bigr) \eqd \bigl( X,\,\ct _{\La }(X)\bigr) .
\end{align*}
\end{exm}



\begin{thebibliography}{99}
\bibitem{bs} A.N.~Borodin, P.~Salminen, Handbook of Brownian Motion -- Facts and Formulae, corrected reprint of 2nd ed., 2002, Birkh\"auser, Basel, 2015. 

\bibitem{buc} R.~Buckdahn, Anticipative Girsanov transformations, Probab.\ Theory Relat.\ Fields {\bf 89} (1991), 211--238.

\bibitem{dmy} C.~Donati-Martin, H.~Matsumoto, M.~Yor, Some absolute continuity relationships for certain anticipative transformations of geometric Brownian motions, Publ.\ Res.\ Inst.\ Math.\ Sci.\ {\bf 37} (2001), 295--326.

\bibitem{har22} Y.~Hariya,  Extensions of Bougerol's identity in law and the associated anticipative path transformations, Stochastic\ Process.\ Appl.\ {\bf 146} (2022), 311--334.

\bibitem{har22a} Y.~Hariya, Invariance of Brownian motion associated with exponential functionals, arXiv:2203.08706 (2022).

\bibitem{kus} S.~Kusuoka, The nonlinear transformation of Gaussian measure on Banach space and its absolute continuity (I), J.\ Fac.\ Sci.\ Univ.\ Tokyo Sect.\ IA Math.\ {\bf 29} (1982), 567--597.

\bibitem{leb} N.N.~Lebedev, Special Functions and Their Applications, Dover, New York, 1972.

\bibitem{myPI} H.~Matsumoto, M.~Yor, An analogue of Pitman's $2M-X$ theorem for exponential Wiener functionals, Part I: A time-inversion approach, Nagoya Math.\ J.\ {\bf 159} (2000), 125--166.

\bibitem{mySI} H.~Matsumoto, M.~Yor, Exponential functionals of Brownian motion, I: Probability laws at fixed time, Probab.\ Surv.\ {\bf 2} (2005), 312--347. 

\bibitem{mySII} H.~Matsumoto, M.~Yor, Exponential functionals of Brownian motion, II: Some related diffusion processes, Probab.\ Surv.\ {\bf 2} (2005), 348--384. 

\bibitem{ram} R.~Ramer, On nonlinear transformations of Gaussian measures, J.\ Funct.\ Anal.\ {\bf 15} (1974), 166--187.

\bibitem{uz94} A.S.~\"Ust\"unel, M.~Zakai, Transformation of the Wiener measure under non-invertible shifts, Probab.\ Theory Relat.\ Fields {\bf 99} (1994), 485--500.

\bibitem{yano} K.~Yano, A generalization of the Buckdahn--F\"ollmer formula for composite transformations defined by finite dimensional substitution, J.\ Math.\ Kyoto Univ.\ {\bf 42} (2002), 671--702.

\bibitem{zz} M.~Zakai, O.~Zeitouni, When does the Ramer formula look like the Girsanov formula?, Ann.\ Probab.\ {\bf 20} (1992),  1436--1440.

\end{thebibliography}
\end{document}